\documentclass[12pt]{amsart}%
\usepackage{graphicx}
\usepackage{amsmath}
\usepackage{amsfonts}
\usepackage{amssymb}
\usepackage[mathscr]{eucal}
\makeatletter
\def\numberwithin#1#2{\@ifundefined{c@#1}{\@nocnterrr}{%
  \@ifundefined{c@#2}{\@nocnterr}{%
  \@addtoreset{#1}{#2}%
  \toks@\expandafter\expandafter\expandafter{\csname the#1\endcsname}%
  \expandafter\xdef\csname the#1\endcsname
    {\expandafter\noexpand\csname the#2\endcsname
     .\the\toks@}}}}
\makeatother
\numberwithin{equation}{section}

\newtheorem{theorem}{Theorem}
\numberwithin{theorem}{section}

\newtheorem{lemma}[theorem]{Lemma}

\newtheorem{proposition}[theorem]{Proposition}
\newtheorem{remark}[theorem]{Remark}
\begin{document}

\title{New proofs of Perelman's theorem on shrinking Breathers in Ricci flow}

\author{Peng Lu}
\address{Department of Mathematics, University of Oregon, Eugene, OR 97403, USA}
\email{penglu@uoregon.edu}
\thanks{P.L. is partially supported by Simons Foundation through Collaboration Grant 229727.}

  \author{Yu Zheng}
   \address{Department of Mathematics, East China Normal University, Shanghai, PR China}
  \email{zhyu@math.ecnu.edu.cn}
\thanks{ Y.Z. is partially supported by CNSF through Grant 11671141. }


\begin{abstract}
We give two new proofs of Perelman's theorem that shrinking breathers of Ricci flow
on closed manifolds are gradient Ricci solitons, using the fact that  the singularity models
of type I solutions are shrinking gradient Ricci solitons and the fact that  non-collapsed type I 
ancient solutions have rescaled limits being shrinking gradient Ricci solitons.

\smallskip


\end{abstract}

\maketitle

\section{\bf Introduction}

In his pioneer paper \cite{Pe02I}, as an application of the  $W$-entropy, G. Perelman 
proves that on closed manifolds the shrinking breathers are shrinking  GRS
 (short for gradient Ricci solitons). 
In this article we give two new proofs of the Perelman's theorem for shrinking breathers
 from the perspective of singularity 
analysis (Theorem \ref{thm main shrinking T}(i) and \ref{thm main shrinking -infty}(i)). 
More precisely,  for shrinking breathers we will construct ancient solutions of Ricci flow which
have type I singularities at finite singular time and $-\infty$,  then we use either the fact that the singularity 
models of type I solutions are shrinking GRS (A. Naber \cite{Na10}, J. Enders,
R. Muller, and P. Topping  \cite{EMT11}), or the fact that non-collapsed type I ancient solutions have
 rescaled limits being shrinking GRS (X.D. Cao and Q. Zhang \cite{CZ11}), to finish the proof.
We can also say a little bit about steady breathers on noncompact manifolds (Proposition
\ref{prop steady noncompact}).

This idea of extending breathers to exist on a larger time interval has been used by 
M. Feldman, T. Ilmanen, and L. Ni (see the Remark in 
\cite[p.53]{FIN05}). They actually use a result of R. Hamilton about asymptotic limit at time $\infty$ 
to give an alternative proof that there is no non-trivial expanding breather on closed manifolds
(\cite{Iv93}, \cite{Pe02I}).

When $M$ is noncompact, under some extra assumption our method also proves 
that shrinking breathers are GRS (Theorem \ref{thm main shrinking T}(ii) and 
\ref{thm main shrinking -infty}(ii)). 
Note that in \cite{Zh14} Q. Zhang proves a result about shrinking breathers on complete noncompact 
manifolds by considering $W$-entropy, which has different assumptions from ours.
Also note that $W$-entropy on noncompact manifolds is used by M. Rimondi and G. Veronelli \cite{RV16}
to prove a result about when a complete noncompact Ricci soliton is a GRS. 

\vskip .2cm
\noindent \textbf{Acknowledgement}. 
P.L. wants to thank Professor Li, jiayu and School of Mathematical Sciences at the University of Science and 
Technology of China, where part of this work is carried out, for their warm hospitality 
during spring, 2017.

\section{\bf Construction of ancient solutions}

Let $(M^n, g(t)), \, t \in (\alpha,\beta)$, be a Ricci flow.
$g(t)$ is called a breather if for some $\tilde{t}_0, \tilde{t}_1 \in (\alpha,\beta)$ with $\tilde{t}_0 
< \tilde{t}_1$ there is a constant $\alpha >0$ and a diffeomorphism $\varphi: M \rightarrow M$
such that the metrics $g(\tilde{t}_1)= \alpha \varphi^* g(\tilde{t}_0)$.
The cases $\alpha < 1, \alpha = 1$, and $ \alpha > 1$ correspond to shrinking, steady, and
expanding breathers, respectively.

\begin{lemma} \label{lem extend sol of RF}
Given a breather $(M^n, g(t))$ we have the following.

\noindent $($i$)$ If the  breather is shrinking, there is an ancient solution of Ricci flow $(M^n,G(t))$
such that $G(t) =g(t)$ when $t \in (\alpha,\beta)$.

\noindent $($ii$)$ If the breather is steady, there is an eternal solution of Ricci flow $(M^n,G(t))$
such that $G(t) =g(t)$ when $t \in (\alpha,\beta)$.

\noindent $($iii$)$ If the breather is expanding, there is an immortal solution of Ricci flow $(M^n,G(t))$
such that $G(t) =g(t)$ when $t \in (\alpha,\beta)$.
\end{lemma}

\begin{proof}
Let $G_0(t) = g(t)$ for $t \in [ \tilde{t}_0, \tilde{t}_1]$. First we define 
$G_1(t)=\alpha \cdot \varphi^* G_0(\alpha^{-1} (t- \tilde{t}_1)+ \tilde{t}_0)$ for
 $t \in [ \tilde{t}_1, \tilde{t}_2]$
where $\tilde{t}_2 = \tilde{t}_1 + \alpha(\tilde{t}_1-\tilde{t}_0)$. 
It is clear that $G_1(t)$ is a solution of Ricci flow. 
Below  we will use the convention that $\frac{1-\alpha^j}{1-\alpha} =j$ when $\alpha =1$.
Inductively  for $j \geq 1$ we define
\[
 G_{j+1}(t)=\alpha \cdot \varphi^* G_j(\alpha^{-1} (t- \tilde{t}_{j+1})+ \tilde{t}_{j}), \quad
t \in [ \tilde{t}_{j+1},  \tilde{t}_{j+2}], 
\]
where  $\tilde{t}_{j+2} = \tilde{t}_{j+1} + \alpha(\tilde{t}_{j+1}-\tilde{t}_j)$,
and for $ j \leq 0$ we define  
\[
G_{j -1}(t)=\alpha^{-1} \cdot (\varphi^{-1})^* G_j (\alpha (t- \tilde{t}_{j-1})+ \tilde{t}_j), 
\quad t \in [ \tilde{t}_{j-1},  \tilde{t}_j],
\]
where $\tilde{t}_{j -1} = \tilde{t}_{j} - \alpha^{-1}(\tilde{t}_{ j+1}-\tilde{t}_j)$.

The following statements follow from some easy calculations. 
For $j \geq 2$ we have $\tilde{t}_j =\tilde{t}_0 +\frac{1-\alpha^j}{1-\alpha} (\tilde{t}_1- \tilde{t}_0) $
and  for $j \geq 1$ we have $\tilde{t}_{-j} =\tilde{t}_1 - \frac{\alpha^{-j-1}-1}{\alpha^{-1}-1} 
(\tilde{t}_1- \tilde{t}_0) $.
  For $j \geq 1$
\begin{align*}
& G_j(t) = \alpha^j \cdot (\varphi^j)^* g \left (\alpha^{-j} \left (t - \tilde{t}_0 - \frac{1-\alpha^{j}}
{1-\alpha}  (\tilde{t}_1- \tilde{t}_0) \right)+ \tilde{t}_0 \right ),   \\
& G_{-j}(t) = \alpha^{-j} \cdot (\varphi^{-j})^* g \left (\alpha^j \left (t - \tilde{t}_1 +\frac{\alpha^{-(j+1)
} -1 }{\alpha^{-1} -1} (\tilde{t}_1- \tilde{t}_0) \right)+ \tilde{t}_0 \right ).
\end{align*}

We define a family of metrics $G(t)$ by piecing  $G_j(t), j \in \mathbb{Z}$, all together. 
By the uniqueness of solutions of Ricci flow or more directly by the fact that the left and right 
time-derivatives of $G(t)$ are matching at each $\tilde{t}_j$,
 we conclude that $G(t)$ is a smooth solution and that $G(t)=g(t)$ for $t \in (\alpha,\beta)$.

\vskip .1cm
\noindent $($i$)$  For $\alpha < 1$, $G(t)$ is defined on $(-\infty , \tilde{t}_0 +\frac{1}{1-\alpha} 
(\tilde{t}_1- \tilde{t}_0))$, which is an ancient solution.

\vskip .1cm
\noindent $($ii$)$ For $\alpha =1$, $G(t)$ is defined on $(-\infty, \infty)$, 
which is an eternal solution.

\vskip .1cm
\noindent $($iii$)$ For $\alpha >1$, $G(t)$ is defined on $(\tilde{t}_1 -\frac{1}{1-\alpha^{-1}}(
\tilde{t}_1 - \tilde{t}_0) , \infty)$, which is an immortal solution.
\end{proof}

Let $\kappa >0$ be a constant.
We call that Ricci flow $(N^n,h(t)), t \in [\alpha, \beta]$, is  $\kappa$-non-collapsed on all scales
if for any point $(p_*, {t}_*) \in N \times  [\alpha, \beta] $ and $r_*>0$
which satisfy $| \operatorname{Rm}|_h \leq r_*^{-2}$ on parabolic ball
 $ \cup_{t \in [\max \{t_* -r_*^2, \alpha \},
t_*]} \left ( B_{h(t)}(p_*, r_*) \times \{t \} \right )$, 
the volume $\operatorname{Vol}_{h(t_*)} B_{h(t_*)}(p_*, r_*)
\geq \kappa r_*^n$.  In this definition we allow $\alpha = -\infty$.
When $\alpha =-\infty$ the definition agrees with \cite[Def. 1.1]{CZ11}.

\begin{lemma} \label{lem G type I}
Let $(M^n, g(t))$ be a nonflat shrinking breather with $g(\tilde{t}_1)= \alpha \varphi^* g(\tilde{t}_0)$
for some $\alpha \in (0,1)$ and $ \tilde{t}_0 < \tilde{t}_1$,
and let $G(t)$ be the solution  constructed in Lemma \ref{lem extend sol of RF}$($i$)$.
Assume  curvature bound $\sup_{(x,t) \in M \times [\tilde{t}_0, \tilde{t}_1]} |
\operatorname{Rm}|_{g}(x,t) < \infty$.
 Then

\vskip .1cm
\noindent $($i$)$ $G(t)$ develops a Type I singularity at time $T =\tilde{t}_0 +\frac{1}{1-\alpha}
 (\tilde{t}_1- \tilde{t}_0)$.

\vskip .1cm
\noindent $($ii$)$  $G(t)$ develops a Type I singularity at time $-\infty$, i.e., 
$\sup_{(x,t) \in M \times (-\infty, \tilde{t}_0]} |t| |\operatorname{Rm}|_{G}$ $(x,t) < \infty$.

\vskip .1cm
\noindent $($iii$)$  If $g(t), \, t \in [\tilde{t}_0, \tilde{t}_1]$, is $\kappa$-non-collapsed on all scales
for some $\kappa >0$,
 then $G(t), \, t \in (- \infty, \tilde{t}_0]$, is $\kappa$-non-collapsed on all scales.
\end{lemma}

\begin{proof}
Define 
\[
K = \sup_{(x,t) \in M \times [\tilde{t}_0, \tilde{t}_1]} |\operatorname{Rm}|_{g}(x,t) \in (0, \infty).
\]

 \vskip .1cm
  $($i$)$ For $j \geq 1$ and $t \in  [ \tilde{t}_j, \tilde{t}_{j+1}]$
   we have
\begin{align}
&  \sup_{x \in M } |\operatorname{Rm}|_{G_j}(x,t) \leq \alpha^{-j} K \leq 
 \frac{ (1-\alpha)^{-1}(\tilde{t}_1 - \tilde{t}_0)K }{T-t}  ,  \label{eq 1a}\\
&\sup_{ t \in  [ \tilde{t}_j, \tilde{t}_{j+1}]}  \sup_{x \in M } |\operatorname{Rm}|_{G_j}(x,t) 
= \alpha^{-j} K . \label{eq 1b}
 \end{align}
Equation (\ref{eq 1b}) implies that $T$ is a singular time. It follows from
(\ref{eq 1a}) that
 \[
  \sup_{x \in M } |\operatorname{Rm}|_{G}(x,t) \leq  \frac{ (1-\alpha)^{-1}
  (\tilde{t}_1 - \tilde{t}_0)K }{T-t}
  \]
 for $t \in [ \tilde{t}_1, T)$, and $G(t)$ develops a Type I singularity at time $T$.

 \vskip .1cm
$($ii$)$ For $j \geq 1$ and $t \in ( \tilde{t}_{-j}, \tilde{t}_{-j+1}]$ we have
\[
 \sup_{x \in M } |\operatorname{Rm}|_{G_{-j}}(x,t) \leq \alpha^{j} K,
 \]
and for $j$ large enough we have
 \[
 \frac{1}{|t|} \geq \frac{1}{|  \tilde{t}_{-j}|}= \frac{1}{ | \tilde{t}_1 - 
 \frac{\alpha^{-(j+1)} -1}{ \alpha^{-1} -1} (\tilde{t}_1 - \tilde{t}_0)  |} \geq 
 \frac{ \alpha^{-1} -1 }{2(\tilde{t}_1 - \tilde{t}_0)} \cdot \alpha^{j+1}.
 \]
 Hence 
 \[
  \sup_{x \in M } |\operatorname{Rm}|_{G}(x,t) \leq  \frac{ 2 (\tilde{t}_1 - \tilde{t}_0)K }{ 
  (1- \alpha)|t|}
  \]
 for $t$ sufficiently close to $-\infty$,  and $G(t)$ develops a Type I singularity at time $-\infty$.

 \vskip .1cm
 $($iii$)$ For any point $(p_*, {t}_*) \in M \times (-\infty, \tilde{t}_0]$, there is a $j_* \geq 1$ such that
 $ {t}_* \in ( \tilde{t}_{-j_*},  \tilde{t}_{-j_* +1}  ]$.
If there is a $r_*>0$ such that  $| \operatorname{Rm}|_G(x,t) \leq r_*^{-2}$ for $ (x,t) \in
 \cup_{t \in [t_* -r_*^2,
t_*]} $ $  \left ( B_{G(t)}(p_*, r_*) \times \{t \} \right )$, by the definition of $G_{-j_*}(t)$ and 
the curvature scaling we have  $ | \operatorname{Rm}|_g \leq  \alpha ^{-j_*}r_*^{-2} $ on 
 \[
  \cup_{t \in [\max \{ \bar{t}_* - \alpha ^{j_*} r_*^{2}, \tilde{t}_0 \}, \bar{t}_*]} ( B_{g(t)}
 (\varphi^{-j_*}(p_*),  \alpha ^{j_* /2}r_*) \times \{t \} ),
 \]
 where $\bar{t}_* =\alpha^{ j_*}(t_* - \tilde{t}_1)+ \tilde{t}_0 + \frac{\alpha^{-(j_*+1)} -1 }
{\alpha^{-1} -1 } (\tilde{t}_1- \tilde{t}_0)$. By the assumption of the $\kappa$-non-collapsing of $g(t)$,
we have
\[
\operatorname{Vol}_{g(\bar{t}_*)} B_{g(\bar{t}_*)} (\varphi^{-j_*}(p_*),  \alpha ^{j_*
/2}r_*)  \geq  \kappa ( \alpha ^{j_* /2}r_*)^n,
\]
hence 
\[
\operatorname{Vol}_{G(t_*) }B_{G(t_*)} (p_*, r_*) \geq  \kappa r_*^n.
\]
This proves that $G(t)$ is  $\kappa$-non-collapsed on all scales. 
Now $($iii$)$  is proved.
\end{proof}

\section{\bf  Proof of main theorems}

In this section we adopt the notations used in Lemma \ref{lem G type I}.
Now we give the first new proof of  Perelman's theorem that on closed manifolds
shrinking breathers are GRS.

\begin{theorem} \label{thm main shrinking T}
Let $(M^n, g(t))$ be a nonflat complete shrinking breather.   
Assume one of the following,

\vskip .1cm
\noindent $($i$)$ $M$ is closed; or

\vskip .1cm
\noindent $($ii$)$ $M$ is noncompact  with bounded curvature 
and there are a strictly increasing subsequence $\{ j_k \}$ of natural numbers
and a sequence $\{ {p}_{j_k} \}$ of points  in $M$ 
such that $($iia$)$  ${p}_{j_k}$ converges to some point ${p}_{\infty}$,
$($iib$)$ $ |\operatorname{Rm} |_{g} (\varphi^{j_k -1}( {p}_{j_k}), \tilde{t}_1) 
\geq c_1$ for  some constant $c_1 >0$,  
 and $($iic$)$ $\varphi^{j_k -1} ({p}_{\infty})$ converges to a point ${p}^{\prime}_{\infty} 
 \in M$;
 
\vskip .1cm
\noindent then $(M, g(t))$ is a shrinking GRS.
\end{theorem}

\begin{proof}
Define 
\[
\lambda_j = \left (T- \left (\tilde{t}_0 +\frac{1-\alpha^{j}}{1-\alpha} (\tilde{t}_1
- \tilde{t}_0) \right )  \right )^{-1} =\alpha ^{-j} (1 -\alpha) (\tilde{t}_1 - \tilde{t}_0 )^{-1}
\rightarrow \infty
\]
 and the rescaled Ricci flow $g_j(t) = \lambda_j G(T + \lambda_j^{-1} t)$. 
 We compute 
 
\begin{equation}
g_j(-1) =\alpha^{-1} (1 -\alpha) (\tilde{t}_1 - \tilde{t}_0 )^{-1} (\varphi^{j-1})^*g(\tilde{t}_1)  
\label{eq rescaled metric g t2}
\end{equation}
and for any $x \in M$
\begin{equation}
 \lambda_j |\operatorname{Rm}|_{g_j}(x,  -1) = |\operatorname{Rm}|_{G}(x, T-\lambda_j^{-1})
= \alpha ^{-j+1} |\operatorname{Rm}|_{g}( \varphi^{j- 1}(x), \tilde{t}_1).
\label{eq rescaled metric curv}
\end{equation}

\vskip .1cm
\noindent  (i) When $M$ is closed, fix $p_* \in M$ where 
$ |\operatorname{Rm}|_{g}(p_*,\tilde{t}_1) \neq 0$,  let $\hat{p}_j =\varphi^{-j+1}(p_*)$ and
$t_j =T- \lambda_j^{-1}$ for $j \geq 2$.
It follows from (\ref{eq rescaled metric curv}) that 
\[
 |\operatorname{Rm}|_{G}(\hat{p}_j, t_j) =\frac{ \alpha (1 -\alpha)^{-1} (\tilde{t}_1 -
 \tilde{t}_0) |\operatorname{Rm}|_{g}
 ( p_*,\tilde{t}_1)}{ T-t_j}\rightarrow  \infty \quad \text{  as } j \rightarrow \infty.
 \]
 Hence $(\hat{p}_j, t_j)$ is an essential blow-up sequence as defined in 
 \cite[Def. 1.2]{EMT11}. 
 Let $\hat{p}_{\infty}$ be a limit of some subsequence of $\{ \hat{p}_j \}$,
 then $\hat{p}_{\infty}$  a Type I singular point as in \cite[Def. 1.2]{EMT11}.
We apply  \cite[Thm 1.1]{EMT11} to $(M, g_j(t), \hat{p}_{\infty})$ 
to conclude that the sequence subconverges 
in Cheeger-Gromov sense  to a normalized nontrivial shrinking  GRS in canonical form. 
It follows from (\ref{eq rescaled metric g t2}) that 
the subsequence of $\{ (M, g_j(-1), \hat{p}_{\infty}) \}$ has the same limit as
the corresponding subsequence $(M, \alpha^{-1} (1 -\alpha) (\tilde{t}_1 -
\tilde{t}_0 )^{-1} g(\tilde{t}_1), \varphi^{j-1}(\hat{p}_\infty))$.

Since $M$ is closed, the subsequence of  $\{ \varphi^{j-1}(\hat{p}_\infty) \}$ subconverges to some
$\tilde{p}_{\infty} \in M$. Hence $(M, \alpha^{-1} (1 -\alpha) (\tilde{t}_1 -
\tilde{t}_0)^{-1} g(\tilde{t}_1), \varphi^{j-1}(\hat{p}_\infty))$ 
subconverges to  $(M, \alpha^{-1} (1 -\alpha) (\tilde{t}_1 - \tilde{t}_0)^{-1} g(\tilde{t}_1), 
\tilde{p}_\infty)$,
it follows from the uniqueness of Cheeger-Gromov limit that $ g(\tilde{t}_1)$ is a shrinking  GRS.

\vskip .1cm
\noindent  (ii) When $M$ is noncompact, 
it follows from (\ref{eq rescaled metric curv}) and the assumption (iia) that 
\[
 |\operatorname{Rm}|_{G}({p}_{j_k}, t_{j_k}) =\frac{ \alpha (1 -\alpha)^{-1} (\tilde{t}_1 -
 \tilde{t}_0) |\operatorname{Rm}
 |_{g} (\varphi^{j_k - 1}( {p}_{j_k}), \tilde{t}_1)}{ T-t_{j_k}} \geq \frac{c_2}{ T-t_{j_k}},
 \]
where $c_2 = \alpha (1 -\alpha)^{-1} (\tilde{t}_1 - \tilde{t}_0) c_1$.
Hence $({p}_{j_k},t_{j_k})$ is an essential blow-up sequence and $p_{\infty}$ is a 
Type I singular point by  the assumption (iib). 
Applying  \cite[Thm 1.1]{EMT11} to $(M, g_{j_k}(t), p_{\infty})$ we conclude that the sequence 
subconverges in Cheeger-Gromov sense  to a normalized nontrivial shrinking  GRS in canonical form. 
It follows from (\ref{eq rescaled metric g t2}) that the subsequence of $\{ (M, g_{j_k}(-1), 
p_{\infty}) \}$ 
has the same limit as the corresponding subsequence of $\{ (M, \alpha^{-1} 
(1 -\alpha) (\tilde{t}_1 - \tilde{t}_0)^{-1} g(\tilde{t}_1), \varphi^{j_k-1}(p_\infty)) \}$.

By the assumption (iic) $(M, \alpha^{-1} (1 -\alpha) (\tilde{t}_1 - \tilde{t}_0
)^{-1} g(\tilde{t}_1), \varphi^{j_k-1}(p_\infty))$ 
subconverges to  $(M, \alpha^{-1} (1 -\alpha) (\tilde{t}_1 - \tilde{t}_0)^{-1} g(\tilde{t}_1), 
{p}^{\prime}_\infty)$.
By the uniqueness of Cheeger-Gromov limit  $ g(\tilde{t}_1)$ is a  shrinking  GRS. 
\end{proof}

\begin{remark}
Suppose that there is a $p_* \in M$ such that  $ |\operatorname{Rm} |_{g} 
(p_*, \tilde{t}_1) \geq c_1$. We assume that $\varphi^{- j_k + 1}( p_*)$ converges 
to some $p_{\infty} \in M$   for some subsequence $j_k$ 
and that $\varphi^{j_k -1}( p_{\infty})$ subconverges to some 
${p}^{\prime}_{\infty} \in M$.
Then the  conditions in Theorem \ref{thm main shrinking T}(ii) are satisfied  
by taking $p_{j_k} = \varphi^{- j_k + 1}(p_*)$.
\end{remark}

Below we give the second new proof of the Perelman's theorem.

\begin{theorem} \label{thm main shrinking -infty}
Let $(M^n, g(t))$ be a nonflat complete shrinking breather. 
Assume one of the following,

\vskip .1cm
\noindent $($i$)$ $M$ is closed; or

\vskip .1cm
\noindent $($ii$)$ $M$ is noncompact  with bounded curvature
 and we assume that $(M, g(t))$, $t \in [\tilde{t}_0, \tilde{t}_1]$, is 
 $\kappa$-non-collapsed on all scale and that for the sequence of points $(q_k, {t}_k)$ in 
 \cite[Thm 4.1]{CZ11}, the sequence  $\varphi^{-j_k}(q_k)$ subconverges to some point $\hat{q}_{\infty}
 \in M$ where $j_k$ is defined by (\ref{eq jk def}) below;

\vskip .1cm
\noindent  then $(M, g(t))$ is a shrinking GRS.
\end{theorem}

\begin{proof}
Note that when $M$ is closed,  nonflat solution $(M, g(t))$, $t \in [\tilde{t}_0, \tilde{t}_1]$, 
is  $\kappa$-non-collapsed on all scales for some small $\kappa >0$. 
Hence  it follows from Lemma \ref{lem G type I} that in either case $($i$)$ or $($ii$)$  solution
$(M, G(t))$ is nonflat type I ancient solution
and is $\kappa$-non-collapsed  on all scales. By \cite[Thm 4.1]{CZ11} there is a sequence of points 
$(q_k, {t}_k)$ with ${t}_k \rightarrow - \infty$ such that  $(M,g_k(t),q_k)$
converges to a nonflat shrinking GRS $(M_{\infty}, g_{\infty}(t), q_{\infty})$.
Here ${g}_k(t) \doteqdot |{t}_k|^{-1} G({t}_k +s| {t}_k|)$.
 Choose $j_k$ so that 
 \begin{equation}
 {t}_k \in ( \tilde{t}_{-j_k},  \tilde{t}_{-j_k+1}], \label{eq jk def}
\end{equation}
then 
\[
 g_k(0) = |{t}_k|^{-1} G( t_k ) = | {t}_k|^{-1} \alpha ^{-j_k} (\varphi^{-j_k})^*
 g( \bar{t}_k) 
\]
where  
\begin{equation}
\bar{t}_k =\alpha^{j_k} \left ( {t}_k -  \tilde{t}_1 + \frac{\alpha^{-(j_k+1)} -1 }
{\alpha^{-1} -1 } (\tilde{t}_1- \tilde{t}_0) \right )+\tilde{t}_0 \in ( \tilde{t}_0,\tilde{t}_1].
\label{eq t k converg}
\end{equation}
It follows from (\ref{eq t k converg}) that
$\bar{t}_k $ subconverges to some $\bar{t}_{\infty} \in [\tilde{t}_0,\tilde{t}_1]$ and
 that $ |{t}_k|^{-1} \alpha ^{-j_k}$ subconverges to a number $c_1 >0$.
Note that  to prove the theorem it suffices to prove the subconvergence of  
$\{ (M, g(\bar{t}_k), \varphi^{-j_k}(q_k)) \}$.

\vskip .1cm
\noindent  $($i$)$ When $M$ is closed,  then $\varphi^{-j_k}(q_k)$ subconverges to $ \bar{q}_{\infty}
\in M $.
We conclude that 
$(M, g(\bar{t}_k), \varphi^{-j_k}(q_k))$ subconveregs to $(M, g(\bar{t}_{\infty}), \bar{q}_{\infty})$.
By the uniqueness of Cheeger-Gromov limit we have that $(M_{\infty}, g_{\infty}(0), q_{\infty})$ 
is isometric to  $(M, c_1 g(\bar{t}_{\infty}), \bar{q}_{\infty})$.
Hence $ g(\bar{t}_{\infty})$ is a nonflat shrinking  GRS.

\vskip .1cm
\noindent  $($ii$)$ When $M$ is noncompact,   by the assumption that $\varphi^{-j_k}(q_k) $ 
subconverges to some $\hat{q}_{\infty}$, we can repeat the proof of $($i$)$
to conclude that   $ g(\bar{t}_{\infty})$ is a  shrinking  GRS.
\end{proof}

\vskip .2cm
G. Perelman proves that steady breathers on  closed  manifolds are steady GRS (actually Einstein), 
the following result is 
of interest when $M$ is  noncompact.

\begin{proposition}\label{prop steady noncompact}
Let $(M^n, g(t)), \, t \in [\tilde{t}_0, \tilde{t}_1]$, 
be a nonflat complete steady breather with bounded curvature. 
Assume that  the solution $g(t)$ has positive curvature operator 
 and that  the scalar curvature $R(\cdot , t)$ attains its supremum over space for each $t$.
Then $g(t)$ is a steady GRS.
\end{proposition}

\begin{proof}
The assumptions enable us to apply \cite[Main Thm]{Ha93} to conclude that $(M, G(t))$  
in Lemma \ref{lem extend sol of RF}(i) is a steady GRS.
\end{proof}


\end{document}